\documentclass[10pt]{article}
\usepackage{graphicx}
\usepackage{tikz}
\usepackage[margin=1in]{geometry}
\usepackage{amsmath, amssymb, amsfonts, amsthm}
\newtheorem{lem}{\noindent {\bf Lemma}}[section]
\newtheorem{prop}{\noindent {\bf Proposition}}[section]
\newtheorem{coro}{\noindent {\bf Corollary}}[section]
\newtheorem{thm}{\noindent {\bf Theorem}}[section]

\newcounter{remark}
\newenvironment{remark}{\smallskip\noindent {\bf Remark \arabic{section}.\arabic{remark}.}}
{\addtocounter{remark}{1}\par}
\newcounter{example}
\catcode`@=11 
\date{}
\newcounter{defi}
\newenvironment{defi}{
\smallskip \noindent
{\bf
  Definition \arabic{section}.\arabic{defi}.
}}{\addtocounter{defi}{1}\par}
\newcommand{\ZZ}{\mathbb Z}
\newcommand{\NN}{\mathbb N}

\newcommand{\U}{\mathcal U}

\title{\bf On metric spaces with given transfinite asymptotic dimensions}
\author{\large  Yan Wu$^\ast$\qquad Jingming Zhu$^{\ast\ast}$\qquad Taras Radul $^{\ast\ast\ast}$
\footnote{
$^\ast$ College of Data Science, Jiaxing University, Jiaxing , 314001, P.R.China.\newline
E-mail: yanwu@mail.zjxu.edu.cn
\newline
$^\ast\ast$ College of Data Science, Jiaxing University, Jiaxing , 314001, P.R.China.\newline
 E-mail: jingmingzhu@mail.zjxu.edu.cn
 \newline
 $^{\ast\ast\ast}$  Institute of Mathematics, Casimirus the Great University of Bydgoszcz, Poland;
\newline
Department of Mechanics and Mathematics, Ivan Franko National University of Lviv,
Universytetska st., 1.79000 Lviv, Ukraine.
\newline
 E-mail: tarasradul@yahoo.co.uk}}
\date{}

\begin{document}
\maketitle
\begin{center}
\begin{minipage}{0.9\textwidth}
\noindent{\bf Abstract.}
For every countable ordinal number $\xi$, we construct a metric space $X_{\xi}$ whose transfinite asymptotic dimension and complementary-finite asymptotic dimension are both $\xi$.
%We also prove that the metric space $X_{\xi}$ has finite decomposition complexity.


\end{minipage}
\end{center}
\footnote{
This research was supported by
the National Natural Science Foundation of China under Grant (No.12071183,11871342)

}

\begin{section}{Introduction}

In coarse geometry, asymptotic dimension of a metric space is an important concept which was defined
by Gromov for studying asymptotic invariants of discrete groups \cite{Gromov}. This dimension can be considered
as an asymptotic analogue of the Lebesgue covering dimension. As a large scale analogue of W.E. Haver's
property $C$ in dimension theory, A. Dranishnikov introduced the notion of asymptotic property $C$ in \cite{Dranishnikov}. It is
well known that every metric space with finite asymptotic dimension has asymptotic property $C$. But the
inverse is not true, which means that there exists some metric space $X$ with infinite asymptotic dimension
and asymptotic property $C$. Therefore how to classify the metric spaces with infinite asymptotic dimension
into smaller categories becomes an interesting problem.

In \cite{Radul2010}, T. Radul defined the transfinite asymptotic dimension (trasdim) which can be viewed as a transfinite extension of the asymptotic dimension and proved that for a metric space $X$, trasdim$(X)<\infty$ if and only if $X$ has asymptotic property $C$.  He also gave examples of metric spaces with trasdim$=\infty$ and with trasdim$=\omega$, where $\omega$ is the smallest infinite ordinal number (see \cite{Radul2010}). But whether there is a metric space $X$ with $\omega<$trasdim$(X)<\infty$ (stated as``omega conjecture"in \cite{satkiewicz} by M. Satkiewicz) was unknown until recently. A metric space $X$ with trasdim$X=\omega+1$ was constructed in \cite{omega+1}. Some examples of metric spaces with higher trasdim were constructed in \cite{omega+k}. Let us remark that trasdim takes only countable values \cite{Radul2010}.
In this paper, for every countable ordinal number $\xi$, we construct a metric space $X_{\xi}$ with trasdim$(X_{\xi})=\xi$ and coasdim$(X_{\xi})=\xi$, which generalized the results in \cite{omega+1} and \cite{omega+k}. Let us remark that this result combined with a recent result in \cite{Orz2020} gives an answer to a problem stated by \cite{Dydak2020}.

%As another generalization of asymptotic dimension, E.Guentner, R.Tessera and G.Yu introduced the notion of finite decomposition complexity to study topological rigidity of manifolds in \cite{Yu2012}. And they proved that
%every metric space with finite asymptotic dimension has finite decomposition complexity in \cite{Yu2013}.
%The relation between asymptotic property $C$ and finite decomposition complexity was studied by A.Dranishnikov and M.Zarichnyi in \cite{Mari2014}. But till now, there is no example of metric space known which make a difference between asymptotic property $C$ and finite decomposition complexity. The metric space $X_{\xi}$ constructed above can not make a difference either. In fact, we prove that $X_{\xi}$ has both asymptotic property $C$ and finite decomposition complexity.

The paper is organized as follows: In Section 2, we recall some definitions and properties of transfinite asymptotic dimension and complementary-finite asymptotic dimension. In
Section 3, we introduce a concrete metric space $X_{\xi}$ with trasdim$(X_{\xi})=\xi$ and coasdim$(X_{\xi})=\xi$ for every countable ordinal number $\xi$.
%Finally, we prove the metric space $X_{\xi}$ has finite decomposition complexity.
\end{section}

\begin{section}{Preliminaries}\

Our terminology concerning the asymptotic dimension follows from \cite{Bell2011} and for undefined terminology we refer to \cite{Radul2010} and \cite{omega+k}.

Let~$(X, d)$ be a metric space and $U,V\subseteq X$, let
\[
\text{diam}~ U=\text{sup}\{d(x,y)~|~ x,y\in U\}
\text{   and   }
d(U,V)=\text{inf}\{d(x,y)~|~ x\in U,y\in V\}.
\]
Let $R>0$ and $\mathcal{U}$ be a family of subsets of $X$. $\mathcal{U}$ is said to be \emph{$R$-bounded} if
\[
\text{diam}~\mathcal{U}\doteq\text{sup}\{\text{diam}~ U~|~ U\in \mathcal{U}\}\leq R.
\]
In this case, $\mathcal{U}$ is said to be \emph{uniformly bounded}.
Let $r>0$, a family $\mathcal{U}$ is said to be\emph{ $r$-disjoint} if
\[
d(U,V)\geq r~\text{for every}~ U,V\in \mathcal{U}\text{~with~}U\neq V.
\]

In this paper, we denote
$\bigcup\{U~|~U\in\mathcal{U}\}$ by $\bigcup\mathcal{U}$, denote $\{U~|~U\in\mathcal{U}_{1}\text{~or~}~U\in\mathcal{U}_{2}\}$
by $\mathcal{U}_{1}\cup\mathcal{U}_{2}$.
Let $A$ be a subset of a metric space $X$ and $\epsilon>0$. We denote $\{x\in X~|~d(x,A)<\epsilon\}$ by $N_{\epsilon}(A)$
and denote $\{x\in X~|~d(x,A)\leq\epsilon\}$ by $\overline{N_{\epsilon}(A)}$.
We denote $\{N_{\epsilon}(U)~|~U\in\mathcal{U}\}$ by $N_{\epsilon}(\mathcal{U})$ and
denote $\{\overline{N_{\epsilon}(U)}~|~U\in\mathcal{U}\}$ by $\overline{N_{\epsilon}(\mathcal{U})}$.

\begin{defi}(\cite{Gromov})
A metric space $X$ is said to have \emph{finite asymptotic dimension} if
there exists $n\in\NN$, such that for every $r>0$,
there exists a sequence of uniformly bounded families
$\{\mathcal{U}_{i}\}_{i=0}^{n}$ of subsets of $X$
such that the family
$\bigcup_{i=0}^{n}\mathcal{U}_{i}$ covers $X$ and each $\mathcal{U}_{i}$
is $r$-disjoint for $i=0,1,\cdots,n$. In this case, we say that
the \emph{asymptotic dimension} of $X$ less than or equal to $n$, which is denoted by
 asdim$(X)\leq n$.

We say that asdim$(X)= n$ if  asdim$(X)\leq n$ and asdim$(X)\leq n-1$ is not true.

\end{defi}

T. Radul generalized asymptotic dimension of a metric space $X$ to transfinite asymptotic dimension which is denoted by trasdim$(X)$ (see \cite{Radul2010}). We will need the set-theoretical function Ord from \cite{Borst1988} which classifies families of finite non-empty subsets of a set $T$. We will consider only the particular case when $T=\NN$.

\begin{defi}\rm(\cite{Borst1988})
Let $Fin\mathbb{N}$ denote the collection of all finite, nonempty
subsets of $\mathbb{N}$ and let $ M \subseteq Fin\mathbb{N}$. For $\sigma\in \{\varnothing\}\cup Fin\mathbb{N}$, let
$$M^{\sigma} = \{\tau\in Fin\mathbb{N} ~|~ \tau \cup \sigma \in M \text{ and } \tau \cap \sigma = \varnothing\}.$$
Let $M^a$ abbreviate $M^{\{a\}}$ for $a \in \NN$. Define\emph{ the ordinal number} Ord$M$ inductively as follows:
\begin{eqnarray*}
\text{Ord}M = 0 &\Leftrightarrow& M = \varnothing,\\
\text{Ord}M \leq \alpha &\Leftrightarrow& \forall~ a\in \mathbb{N}, ~\text{Ord}M^a < \alpha,\\
\text{Ord}M = \alpha &\Leftrightarrow& \text{Ord}M \leq \alpha \text{ and } \text{Ord}M < \alpha \text{ is not true},\\
\text{Ord}M = \infty &\Leftrightarrow& \text{Ord}M \leq\alpha \text{ is not true for every ordinal number } \alpha.
\end{eqnarray*}

\end{defi}
\rm We call a family $M \subseteq Fin \NN$ \emph{inclusive} if and only if for each $\sigma, \tau\in Fin \mathbb{N}$ such that $\tau\subseteq\sigma$ and
\rm $\sigma\in M$, we have $\tau\in M$. We consider only inclusive families in the following.

\rm The following lemmas are particular cases of the corresponding lemmas from \cite{Borst1988}.

\begin{lem}\rm(\cite{Borst1988})
\label{lemord0}
Let $M \subseteq Fin \NN$ and $n\in\NN$.
Then\[ \text{Ord$ M\leq n$ if and only if $|\sigma|\leq n$ for every $\sigma \in M$.}\]
\end{lem}
\begin{lem}\rm(\cite{Borst1988})
\label{lemordsubset}
If $ M,N\subseteq Fin\mathbb{N}$ and $M\subseteq N$, then $\text{Ord}M\leq\text{Ord}N.$
\end{lem}
\begin{lem}\rm(\cite{Borst1988})
\label{lemord2}
If $ M,N\subseteq Fin\mathbb{N}$, then $\text{Ord}(M\cup N)\leq\text{max}\{\text{Ord}M,\text{Ord}N\}.$
\end{lem}

\begin{lem}\rm(\cite{Borst1988})
\label{lemord1}
Let $L$ and $L'$ be sets. Let $Fin L$ and $Fin L'$ denote  collections of all finite, nonempty
subsets of $L$ and $L'$, respectively.
Let $M \subseteq Fin L$, $M'\subseteq Fin L'$ and
$\phi: L\to L'$ be a function  such that for every $\sigma \in M$, we have $\phi(\sigma)\in M'$ and $|\phi(\sigma)|=|\sigma|$. Then Ord$ M\leq$Ord $M'$.
\end{lem}

Let $M\subseteq Fin\mathbb{N}$ and $K$ be an infinite subset of $\mathbb{N}$. Then there is a standard bijection $\varphi$ from $\mathbb{N}$ to $K$ which keeps the order. We define $M[K]=\Big\{\{\varphi(k_1),\varphi(k_2),\cdots,\varphi(k_m)\}~\Big|~\{k_1,k_2,\cdots,k_m\}\in M\Big\}$.
Note that $M=M[\mathbb{N}]$.
By Lemma \ref{lemord1}, we obtain the following result.
\begin{coro}\rm
\label{coroord}
 Let $M\subseteq Fin\mathbb{N}$ and $K$ is a infinite subset of $\mathbb{N}$, then Ord$M=$Ord$M[K]$.
\end{coro}

\begin{defi}\rm(\cite{Radul2010})
Given a metric space $X$, define the following collection:
\[
\begin{split}
A(X,d) = \{\sigma \in Fin\mathbb{N}~ |~&\text{ there are no uniformly bounded families } \mathcal{U}_i  \text{ for } i \in \sigma
 \\& \text{ such that each } \mathcal{U}_i
\text{ is } i\text{-disjoint and }\bigcup_{i\in\sigma}\mathcal{U}_i \text{~covers~} X\}
\end{split}\]
and
\[
\begin{split}
A_{2}(X,d) = \{\sigma\in Fin \mathbb{N}~  | ~&\text{ there are no uniformly bounded families } \mathcal{V}_i \text{ for }i \in \sigma \\
&\text{ such that each }\mathcal{V}_i\text{ is } 2^{i}\text{-disjoint and }\bigcup_{i\in\sigma}\mathcal{V}_i \text{ covers } X  \}.
\end{split}
\]
The \emph{transfinite asymptotic dimension} of $X$ is defined as trasdim$(X) =$ Ord$A(X,d)$ .
\end{defi}

\begin{lem}\rm
\label{lemordequal}

Ord$A(X,d)=$Ord$A_{2}(X,d)$.

\end{lem}
\begin{proof}
Since $A(X,d)\subseteq A_{2}(X,d)$, Ord$A(X,d)\leq$Ord$A_{2}(X,d)$ by Lemma \ref{lemordsubset}.
The inverse inequality follows from Lemma \ref{lemord1} by considering the exponential function $\phi: \mathbb{N}\to \mathbb{N}$ defined by $\phi(n)=2^{n}$.
\end{proof}

 \begin{defi}(\cite{yanzhu2018})
Every ordinal number $\gamma$ can be represented as $\gamma=\lambda(\gamma)+n(\gamma)$, where $\lambda(\gamma)$ is the limit ordinal or $0$ and $n(\gamma)\in \mathbb{N}\cup \{0\}$. Let $X$ be a metric space, we define \emph{complementary-finite asymptotic dimension} of $X$ (\text{coasdim}$(X)$) inductively as follows:
\begin{itemize}
\item\text{coasdim}$(X)=-1$ $\Leftrightarrow$ $X=\emptyset$,

\item \text{coasdim}$(X)\leq \lambda(\gamma)+n(\gamma)$ $\Leftrightarrow$ for every $r>0$, there exist $r$-disjoint uniformly bounded families $\mathcal{U}_0,\cdots,\mathcal{U}_{n(\gamma)}$ of subsets of $X$ such that \text{coasdim}$(X\setminus \bigcup\bigcup_{i=0}^{n(\gamma)}\mathcal{U}_i)<\lambda(\gamma)$,

\item \text{coasdim}$(X)=\gamma$ $\Leftrightarrow$ \text{coasdim}$(X)\leq \gamma$ and \text{coasdim}$(X)\leq \beta$ is not true
for any $\beta<\gamma$,

\item \text{coasdim}$(X)=\infty$ $\Leftrightarrow$ \text{coasdim}$(X)\leq \gamma$ is not true for any ordinal number $\gamma$.
\end{itemize}
$X$ is said to have \emph{complementary-finite asymptotic dimension} if coasdim$(X)\leq \gamma$ for some ordinal number $\gamma$.

\end{defi}
\begin{remark}
\label{remarkCOASDIMandASDIM}
It is easy to see that for every $n\in\NN \cup \{0\}$, coasdim$(X) \leq n$ if and only if asdim$(X) \leq n$.
\end{remark}

\begin{lem}\rm(\cite{yanzhu2018})
\label{asdimunion}
Let $X$ be a metric space with $X_{1}, X_{2}\subseteq X$. Then
\[\text{coasdim$(X_{1}\cup X_{2})\leq\text{max}\{\text{coasdim}(X_{1}),\text{coasdim}(X_{2})\}$.}
\]
\end{lem}

\end{section}

\begin{section}{Main result}
\subsection{Families with given Ord}
For each $\tau=\{k_0,k_1,\cdots,k_s\}\in Fin\mathbb{N}$, we choose an indexation such that $k_0<k_1<\cdots<k_s$. For every $n\in\mathbb{N}\cup\{0\}$,  let
\[
\text{ $K(\tau,n)=\{k_{n+1},k_{n+2}\cdots,k_s\}$ if $n<s$ and $K(\tau,n)=\varnothing$ if $n\geq s$,}\]
and let\[
 \text{ $i(\tau,n)=\text{max}\{k_{0},k_{1}\cdots,k_n\}=k_n$ if $n\leq s$ and $i(\tau,n)=k_s$ if $n> s$.}\]

For each limit ordinal $\alpha$, we fix an increasing sequence $\{\zeta_{i}(\alpha) + i\}$ of ordinals such that
\[
 \text{each $\zeta_{i}(\alpha)$ is a limit ordinal or 0 and $\alpha=\text{sup}_{i}(\zeta_{i}(\alpha) + i)$.}\]
When $\alpha=\text{sup}_{i}(\beta + i)$ for some limit ordinal $\beta$, we put $\zeta_{i}(\alpha)=\beta$ for each $i\in\mathbb{N}$.

For each countable ordinal number $\xi$, we write $\xi=\gamma(\xi)+n(\xi)$, where $\gamma(\xi)$ is a limit ordinal or 0 and
$n(\xi)\in \mathbb{N}\cup\{0\}$.
We build a family $S_{\xi}\subseteq Fin \mathbb{N}$ by induction.

\begin{defi}
Let $n\in\NN$ and let $\xi$ be a countable ordinal number , we define
\begin{itemize}
\item $S_n=\{\sigma\in Fin \mathbb{N}~|~|\sigma|\leq n\}$.
\item $S_{\xi}=S_{\gamma(\xi)+n(\xi)}=\Big\{\sigma\in Fin \mathbb{N}~\Big|~K(\sigma,n(\xi))\in S_{\zeta_{l}(\gamma(\xi))+l}\cup\varnothing~\text{~for~some }l\in\{1,2,\cdots,i(\sigma,n(\xi))\}\Big\}.$
\end{itemize}
\end{defi}

\begin{remark}\label{remarkadd}
\begin{itemize}
\item It follows from the definition of the family $S_{\xi}$ that  $\sigma\in S_\xi$ for each $\sigma\in Fin \mathbb{N}$ with $|\sigma|\le n(\xi)+1$.
\item Note that if $\sigma=\{k_0,\cdots,k_s\}\in S_{\xi}$, then $\{l,k_0,\cdots,k_s\}\in S_{\xi+1}$ for every $l<k_0$ with $l\in \NN$.
\end{itemize}
\end{remark}

\begin{lem}\rm
\label{lemsmall}
Let $\xi=\gamma(\xi)+n(\xi)$ be an infinite ordinal number and $\tau=\{k_0,\cdots,k_s\}\in S_{\xi}$. Then we have
\[\tau\in S_{\zeta_l(\gamma(\xi))+l+n(\xi)+1}\text{ for some }l\in\{1,2,\cdots,i(\tau,n(\xi))\}.\]
\end{lem}
\begin{proof}
Since $\tau=\{k_0,\cdots,k_s\}\in S_{\xi}$, we have $K(\tau,n(\xi))\in S_{\zeta_{l}(\gamma(\xi))+l}\cup\varnothing~\text{~for~some }l\in\{1,2,\cdots,i(\tau,n(\xi))\}$.

\begin{itemize}
\item If $K(\tau,n(\xi))\in S_{\zeta_{l}(\gamma(\xi))+l},$
then $\tau=\{k_0,\cdots,k_{n(\xi)},k_{n(\xi)+1},\cdots,k_s\}\in S_{\zeta_{l}(\gamma(\xi))+l+n(\xi)+1}$ by Remark 3.2.
\item If $K(\tau,n(\xi))=\varnothing$, then
$K(\tau,l+n(\xi)+1)\subseteq K(\tau,n(\xi))$ implies $K(\tau,l+n(\xi)+1)=\varnothing$. So $\tau\in S_{\zeta_l(\gamma(\xi))+l+n(\xi)+1}$.
\end{itemize}
\end{proof}
\begin{lem}\rm
\label{leminclusive}
The family $S_\xi$ is inclusive for each countable ordinal $\xi$.
\end{lem}

\begin{proof}
\begin{itemize}
\item It follows from the definition of $S_\xi$ that the result is true when $\xi$ is finite.
\item Assume that the result holds for every infinite ordinal number $\xi<\alpha$.
Now for $\xi=\alpha$,  let $\tau=\{k_{j_0},k_{j_1},\cdots,k_{j_t}\}\subseteq\sigma=\{k_0,\cdots,k_s\}\in S_{\xi}$.
Note that $t\leq s$ and $K(\tau,n(\xi))\subseteq~K(\sigma,n(\xi))$. If $t\leq n(\xi)$, then we have $\tau\in S_{\xi}$ by definition.
Consider the case $t> n(\xi)$. Then we have  $s> n(\xi)$ and
\[
K(\sigma,n(\xi))\in S_{\zeta_{l}(\gamma(\xi))+l}~\text{~for~some }1\leq l\leq i(\sigma,n(\xi))= k_{n(\xi)}\leq k_{j_{n(\xi)}}=i(\tau,n(\xi)).
\]
 By inductive assumption, we have $
K(\tau,n(\xi))\in S_{\zeta_{l}(\gamma(\xi))+l}$.
So  $\tau \in S_{\gamma(\xi)+n(\xi)}$.
\end{itemize}

\end{proof}
\begin{lem}\rm
\label{lemsubset}
$S_{\gamma+n}\subseteq S_{\gamma+m}$ , where $\gamma$ is a limit ordinal and $ m\in\NN,n\in\NN\cup\{0\}$ such that $n<m$.
\end{lem}

\begin{proof}
For every $\sigma=\{k_0,\cdots,k_s\}\in S_{\gamma+n}$.
\begin{itemize}
\item If $s\leq m$,  then we have $\sigma\in S_{\gamma+m}$ by definition.
\item If $s>m$, then $s>n$. It follows that
\[\text{$K(\sigma,n)\in S_{\zeta_{l}(\gamma)+l}~\text{~for~some }1\leq l\leq i(\sigma,n)=k_{n}< k_{m}=i(\sigma,m).$}\]
Since $n<m$, we have
$
K(\sigma,m)\subseteq K(\sigma,n).$
Since $S_{\zeta_{l}(\gamma)+l}$ is inclusive by Lemma \ref{leminclusive},
we have
\[
K(\sigma,m)\in S_{\zeta_{l}(\gamma)+l}\text{~for~some }l\in\{1,2,\cdots, i(\sigma,m)\}.\]
Hence $\sigma\in S_{\gamma+m}.$
\end{itemize}
\end{proof}

\begin{lem}\rm
\label{lemlessthan}
 For each countable ordinals $\xi$, Ord$S_{\xi}\leq \xi$.
 \end{lem}

\begin{proof}
 By Lemma \ref{lemord0}, the result is true when $\xi$ is finite. Assume that the result holds for every $\xi<\alpha$. Now $\xi=\alpha$.
For every $a\in\NN$ and  $\sigma\in S_{\xi}^a$,  $\{a\}\sqcup \sigma \in S_{\xi}.$ It follows that
\[K(\{a\}\sqcup \sigma,n(\xi))\in S_{\zeta_{l}(\gamma(\xi))+l}\cup\varnothing~\text{for~some } l\in\{1,2,\cdots, i(\{a\}\sqcup\sigma,n(\xi))\}.\]
\begin{itemize}
\item When $n(\xi)\in\NN$.
\begin{itemize}
 \item If $i(\{a\}\sqcup\sigma,n(\xi))>a$, then $K(\sigma, n(\xi)-1)=K(\{a\}\sqcup \sigma, n(\xi))\in S_{\zeta_l(\gamma(\xi))+l}\cup\varnothing$.
Note that\[ 1\leq l\leq i(\{a\}\sqcup\sigma,n(\xi))=i(\sigma, n(\xi)-1)\text{~which implies that~} \sigma\in S_{\gamma(\xi)+n(\xi)-1}.\]
\item If $i(\{a\}\sqcup\sigma,n(\xi))\leq a$, then
$l\leq a\text{ and }K(\sigma, n(\xi))\subseteq K(\{a\}\sqcup \sigma, n(\xi)).$
Since $ S_{\zeta_{l}(\gamma(\xi))+l}$ is inclusive, $K(\sigma, n(\xi))\in S_{\zeta_{l}(\gamma(\xi))+l}\cup\varnothing$.
So by Remark 3.2, $\sigma\in S_{\zeta_l(\gamma(\xi))+l+n(\xi)+1}$ for some $ l\in\{1,2,\cdots, a\}$.
\end{itemize}
Therefore,\[S_{\xi}^a\subseteq  S_{\gamma(\xi)+n(\xi)-1}\cup(\bigcup_{l=1}^{a} S_{\zeta_l(\gamma(\xi))+n(\xi)+l+1}) .\]
Then by Lemma \ref{lemord2} and inductive assumption,
\[
\text{Ord}(S_{\xi}^a)\leq \gamma(\xi)+n(\xi)-1<\xi.
\]
\item When $n(\xi)=0$.
\begin{itemize}
\item If $i(\{a\}\sqcup\sigma,0)=a$, then $\sigma=K(\{a\}\sqcup \sigma,0)\in S_{\zeta_{l}(\gamma(\xi))+l}\cup\varnothing$, for some $l\in\{1,2,\cdots, a\}$.
\item If $i(\{a\}\sqcup\sigma,0)<a$, then \[
\text{$K(\sigma,0)\subseteq K(\{a\}\sqcup \sigma,0)\in S_{\zeta_{l}(\gamma(\xi))+l}\cup\varnothing~\text{for~some } l\in\{1,2,\cdots,i(\{a\}\sqcup\sigma,0)\}$, }\]which implies $\sigma\in S_{\zeta_{l}(\gamma(\xi))+l+1}$ for some $l< a$, by Lemma \ref{leminclusive}.
 \end{itemize}
So by Lemma \ref{lemsubset},\[S_{\xi}^a\subseteq \bigcup_{l=1}^{a} S_{\zeta_l(\gamma(\xi))+l+1} .\]
Then by Lemma \ref{lemord2} and inductive assumption,
\[
\text{Ord}(S_{\xi}^a)\leq\zeta_a(\gamma(\xi))+a+1<\xi.
\]

\end{itemize}
So in both cases, we have Ord$S_{\xi}\leq \xi$.

\end{proof}

Let $\sigma=\{k_0,\cdots,k_s\}\in Fin \NN$ and $\tau=\{p_0,\cdots,p_s\}\in Fin \NN$ .
We denote $\sigma\leq\tau$ if for each $i\in\{0,1,\cdots,s\}$, $k_i\leq p_i$.

\begin{lem}\rm
\label{lemgreat}
Let $\xi$ be an countable ordinal number and $\sigma\in S_{\xi}$. Then for each $\tau\in \text{Fin} \NN$ such that
$\sigma\leq\tau$, we have $\tau\in S_{\xi}$.
\begin{proof}
 It is easy to see the result is true when $\xi$ is finite. Assume that the result holds for every countable ordinal number $\xi<\alpha$. For $\xi=\alpha$, let $\sigma=\{k_0,\cdots,k_s\}\in S_{\xi}$ and $\tau=\{p_0,\cdots,p_s\}$ such that $k_i\leq p_i$ for $i\in\{0,1,\cdots,s\}$.
\begin{itemize}
\item If $s\leq n(\xi)$, then $\tau\in S_{\xi}$ by Remark 3.2.
 \item If $s> n(\xi)$, then $K(\sigma,n(\xi))\in S_{\zeta_{l}(\xi)+l} $ for some $1\leq l\leq i(\sigma,n(\xi))= k_{n(\xi)}\leq p_{n(\xi)}=i(\tau,n(\xi))$.
Note that $K(\sigma,n(\xi))\leq K(\tau,n(\xi))$.   By inductive assumption,\[ \text{$K(\tau,n(\xi))\in S_{\zeta_{l}(\xi)+l} $ for some $l\in\{1,2,\cdots,i(\tau,n(\xi))\}$.}\]
So $\tau\in S_{\xi}$.
\end{itemize}
\end{proof}
\end{lem}

\begin{prop}\rm
\label{propord}
 Ord$S_{\xi}= \xi$ for each countable ordinal number $\xi$.
\end{prop}

\begin{proof}
  By Lemma \ref{lemord0}, it is easy to obtain that the result is true when $\xi$ is finite.  Assume that the result holds for every countable ordinal number  $\xi<\alpha$. For $\xi=\alpha$.
By Lemma \ref{lemlessthan}, it suffices to show that Ord$S_{\xi}\geq\xi$.

\begin{itemize}

\item If $n(\xi)=0$, then it means that $\xi$ is a limit ordinal.  For every $k\in\NN$, let \[\text{$S_{\zeta_{k}(\xi)+k}(k)=\{\tau\in S_{\zeta_{k}(\xi)+k}\mid t>k$ for each $t\in\tau\}$.}\]  Then for each $\tau\in S_{\zeta_{k}(\xi)+k}(k)$, $\{k\}\sqcup \tau\in S_\xi$ by definition. Since $S_\xi$ is inclusive by Lemma \ref{leminclusive}, $\tau\in S_\xi$. So $S_{\zeta_{k}(\xi)+k}(k)\subseteq S_\xi$.

    Note that\[ \text{Ord} S_{\zeta_{k}(\xi)+k}=\text{Ord} S_{\zeta_{k}(\xi)+k}(k). \]
    Indeed, by Lemma \ref{lemordsubset}, $S_{\zeta_{k}(\xi)+k}(k)\subseteq S_{\zeta_{k}(\xi)+k}$ implies $\text{Ord}  S_{\zeta_{k}(\xi)+k}(k)\leq \text{Ord}S_{\zeta_{k}(\xi)+k}.$ There is a function $\phi:\NN\rightarrow\NN$ defined by $\phi(i)=i+k$. Then by Lemma \ref{lemgreat},\[\text{ $\phi(\sigma)\in S_{\zeta_{k}(\xi)+k}(k)$ whenever $\sigma\in S_{\zeta_{k}(\xi)+k}$ and $|\phi(\sigma)|=|\sigma|$}.\]
     By Lemma \ref{lemord1}, $\text{Ord}  S_{\zeta_{k}(\xi)+k}\leq \text{Ord}S_{\zeta_{k}(\xi)+k}(k).$ So $\text{Ord} S_{\zeta_{k}(\xi)+k}=\text{Ord} S_{\zeta_{k}(\xi)+k}(k).$

    By Lemma \ref{lemordsubset} and inductive assumption,
    \[\text{Ord} S_\xi\ge\text{Ord}  S_{\zeta_{k}(\xi)+k}(k)=\text{Ord}  S_{\zeta_{k}(\xi)+k}=\zeta_{k}(\xi)+k \text{~for each~}k\in\NN.
    \]
   Hence $\text{Ord}  S_\xi\ge \xi$.

\item If $n(\xi)\in\NN$, then for every $k\in\NN$, let \[\text{$S_{\xi-1}(k)=\{\tau\in S_{\xi-1}\mid t>k$ for each $t\in\tau$\}.}\]Then for each $\tau\in S_{\xi-1}(k)$, by Remark 3.2, $\{k\}\sqcup \tau\in S_\xi$ which implies $\tau\in S_\xi^{k}$. Then $S_{\xi-1}(k)\subseteq S_\xi^{k}$. By the similar argument and inductive assumption,
    \[
    \text{Ord}S_\xi^{k}\geq \text{Ord}S_{\xi-1}(k)=\text{Ord} S_{\xi-1}=\xi-1.
    \]
    Hence $\text{Ord}  S_\xi\ge \xi$.

\end{itemize}

\end{proof}

\subsection{Construction of spaces $X_{\xi}$}

\begin{defi}(\cite{Engelking})
Let $X$ be a metric space and let $A,B$ be a pair of disjoint subsets of $X$. We say that a subset $L\subseteq X$ is a \emph{partition}
of $X$ between $A$ and $B$ if there exist open sets $U,W\subseteq X$ satisfying the following conditions
$$A\subseteq U, B\subseteq W\text{ and }X=U\sqcup L\sqcup W.$$
\end{defi}

\begin{defi}(\cite{omega+k})
Let $X$ be a metric space and let $A,B$ be a pair of disjoint subsets of $X$. For every $\epsilon>0$, we say that a subset $L\subseteq X$ is an \emph{$\epsilon$-partition}
of $X$ between $A$ and $B$ if there exist open sets $U,W\subseteq X$ satisfying the following conditions
$$A\subseteq U, B\subseteq W, X=U\sqcup L\sqcup W, d(L,A)>\epsilon\text{ and }d(L,B)>\epsilon$$
Clearly, an $\epsilon$-partition $L$ of $X$ between $A$ and $B$ is a partition of $X$ between $A$ and $B$.
\end{defi}
\begin{lem}\rm(\cite{Engelking}, Lemma 1.8.14)
\label{partition}
Let $F_i^{+}$, $F_i^{-}$, where $i=1,2,\ldots,n$, be the pairs of opposite faces of $I^n\doteq [0,1]^n$. If $I^n=L_0\supset L_1\supset \ldots\supset L_n$ is a decreasing sequence of closed sets such that $L_{i+1}$ is a partition of $L_{i}$ between $L_{i}\cap F_{i+1}^{+}$ and $L_{i}\cap F_{i+1}^{-}$ for $i\in\{0,1,2,\ldots,n-1\}$, then $L_{n}\neq \varnothing$.
\end{lem}
\begin{lem}\rm
\label{partition2}
Let $L_0\doteq [0,B]^n$ for some $B>0$, $F_i^{+}$, $F_i^{-}$, where $i=1,2,\cdots,n$, be the pairs of opposite faces of $L_0$ and let $0<\epsilon<\frac{1}{6}B$.
For $k=1,2,\cdots, n$, let
$\mathcal{U}_k$ be an $\epsilon$-disjoint and $\frac{1}{3}B$-bounded family of subsets of $[0,B]^n$. Then there exists an $\epsilon$-partition $L_{k+1}$ of $L_{k}$ between $F_{k+1}^+\cap L_{k}$ and $F_{k+1}^-\cap L_{k}$ such that $L_{k+1}\subseteq L_{k}\cap(\bigcup \mathcal{U}_{k+1})^c$ for $k=0,1,2,\cdots,n-1$. Moreover, $L_{n}\neq \varnothing$.
\end{lem}
\begin{proof}
For $k=1,2,\cdots, n$, let $\mathcal{A}_k\doteq \{U\in\mathcal{U}_k~|~d(U,F_k^+)\leq2\epsilon\}$ and  $\mathcal{B}_k\doteq \{U\in\mathcal{U}_k~|~ d(U,F_k^+)>2\epsilon\}$.
Note that $\mathcal{A}_k\cup\mathcal{B}_k=\mathcal{U}_k$.
Let
\[A_k=\bigcup\{N_{\frac{\epsilon}{3}}(U)~|~U\in\mathcal{A}_k\}\text{ and }B_k=\bigcup\{N_{\frac{\epsilon}{3}}(U)~|~U\in\mathcal{B}_k\}.\]
Then $d(A_k,F_k^-)>B-\frac{1}{3}B-2\epsilon-\frac{\epsilon}{3}>\frac{4}{3}\epsilon$ and $d(B_k,F_k^+)>2\epsilon-\frac{\epsilon}{3}>\frac{4}{3}\epsilon$. It follows that
\[(A_k\cup N_{\frac{4}{3}\epsilon}(F_k^+))\cap (B_k\cup N_{\frac{4}{3}\epsilon}(F_k^-))=\varnothing.\]
Let $L_{k+1}\doteq L_{k}\setminus ((A_{k+1}\cup N_{\frac{4}{3}\epsilon}(F_{k+1}^+))\cup (B_{k+1}\cup N_{\frac{4}{3}\epsilon}(F_{k+1}^-)))$ for $k=0,1,2,\cdots,n-1$.

Note that  $F_{k+1}^+\cap L_{k}\neq\varnothing$ and $F_{k+1}^-\cap L_{k}\neq\varnothing$ for $k=0,1,\cdots,n-1.$
Indeed,
\begin{itemize}
\item Clearly, $F_1^+\cap L_{0}\neq\varnothing$ and $F_1^-\cap L_{0}\neq\varnothing$.
\item For $k=1,2,\cdots, n-1$,
\[
\begin{split}
L_{k}&=L_{k-1}\setminus ((A_k\cup N_{\frac{4}{3}\epsilon}(F_k^+))\cup (B_k\cup N_{\frac{4}{3}\epsilon}(F_k^-)))\\
&=L_{k-2}\setminus ((A_{k-1}\cup N_{\frac{4}{3}\epsilon}(F_{k-1}^+))\cup (B_{k-1}\cup N_{\frac{4}{3}\epsilon}(F_{k-1}^-))\cup(A_k\cup N_{\frac{4}{3}\epsilon}(F_k^+))\cup (B_k\cup N_{\frac{4}{3}\epsilon}(F_k^-)))\\
&=\cdots\\
&=L_{0}\setminus ((A_{1}\cup N_{\frac{4}{3}\epsilon}(F_{1}^+))\cup (B_{1}\cup N_{\frac{4}{3}\epsilon}(F_{1}^-))\cup\cdots\cup(A_k\cup N_{\frac{4}{3}\epsilon}(F_k^+))\cup (B_k\cup N_{\frac{4}{3}\epsilon}(F_k^-))).
\end{split}
\]
Note that
$F_{k+1}^+\setminus(N_{\frac{4}{3}\epsilon}(F_1^+))\cup N_{\frac{4}{3}\epsilon}(F_1^-)\cup\cdots\cup N_{\frac{4}{3}\epsilon}(F_k^+)\cup N_{\frac{4}{3}\epsilon}(F_k^-))$ is a nonempty connected set with diameter greater than $\frac{5}{9}B$ and
$(A_1\cup B_1\cup\cdots A_k\cup B_k)$ is a union set of a $\frac{\epsilon}{3}$-disjoint and $\frac{4}{9}B$-bounded family. Then
 \[
 F_{k+1}^+\setminus((A_{1}\cup N_{\frac{4}{3}\epsilon}(F_{1}^+))\cup (B_{1}\cup N_{\frac{4}{3}\epsilon}(F_{1}^-))\cup\cdots\cup(A_k\cup N_{\frac{4}{3}\epsilon}(F_k^+))\cup (B_k\cup N_{\frac{4}{3}\epsilon}(F_k^-)))\neq\varnothing.
 \]
It follows that $F_{k+1}^+\cap L_{k}\neq\varnothing $. Similarly, $F_{k+1}^-\cap L_{k}\neq\varnothing$.
\end{itemize}
Therefore, $L_{k+1}$ is an $\epsilon$-partition of $L_{k}$ between $F_{k+1}^+\cap L_{k}$ and $F_{k+1}^-\cap L_{k}$ such that $L_{k+1}\subset L_{k}\cap(\bigcup \mathcal{U}_{k+1})^c$.
By Lemma~\ref{partition}, $L_{n}\neq \varnothing$.

\end{proof}

Let $L=\{n+2~|~n\in\mathbb{N}\}$ and $S_{\xi}[L]=\Big\{\{k_{0}+2,k_{1}+2,\cdots,k_{s}+2\}~\Big|~\{k_{0},k_{1},\cdots,k_{s}\}\in S_{\xi}\Big\}$.
Then by Corollary \ref{coroord}, Ord$S_{\xi}[L]$=Ord$S_{\xi}=\xi$ for each countable ordinal number $\xi$.

\begin{defi}\rm
For $\tau=\{k_0+2,k_1+2,\cdots,k_m+2\}\in S_{\xi}[L]$, we define \[
\text{$X_{\tau}=\Big\{(x_i)_{i=0}^m\in (2^{k_0}\mathbb{Z})^{m+1}~\Big|~|\{j~|~x_j\notin 2^{k_p}\mathbb{Z}\}|\leq p \text{ for every } p\in\{0,1,\cdots, m\}\Big\}$.}\]
We consider $X_\tau$ with sup-metric.
\end{defi}

\begin{prop}\rm
\label{prop1}
$\tau\in A_{2}(X_{\tau},d)$ for every $\tau=\{k_0+2,k_1+2,\cdots,k_m+2\}\in S_{\xi}[L]$.
\end{prop}

\begin{proof}
Suppose that $\tau=\{k_0+2,k_1+2,\cdots,k_m+2\}\notin A_{2}(X_{\tau},d)$, then there are $B$-bounded families $\mathcal{U}_0,\mathcal{U}_1,\cdots,\mathcal{U}_m$ such that $\mathcal{U}_j$ is $2^{k_j+2}$-disjoint for every $j\in\{0,1,\cdots, m\}$ and $\bigcup_{i=0}^m\mathcal{U}_i$ covers $[0,8B]^{m+1}\cap X_{\tau}$ for some $B>2^{k_m+1}$.

Assume that $p_{1}=\frac{8B}{2^{k_m}}\in \mathbb{N}$.
Take a bijection $\psi:\{1,2,\cdots,p_{1}^{m+1}\} \to \{0,1,2,\cdots,p_{1}-1\}^{m+1}$. Let
$$Q(t)=\prod_{j=1}^{m+1}[2^{k_m}\psi(t)_{j},2^{k_m}(\psi(t)_{j}+1)],\text{where }\psi(t)_{j}\text{ is the $j$th coordinate of }\psi(t).$$
Let $\mathcal{Q}_1=\{Q(t)~|~t\in\{1,2,\cdots,p_{1}^{m+1}\}\}$, then $[0,8B]^{m+1}=\bigcup_{Q\in\mathcal{Q}_1}Q$. Note that
 \[\text{$[0,8B]^{m+1}\cap X_{\tau}\subseteq\bigcup_{Q\in \mathcal{Q}_1}\partial_{m}Q$, where $\partial_{m}Q$ is the $m$-dimensional skeleton of $Q$.}\]

Let $L_{0}=[0,8B]^{m+1}$.
Since $\overline{N_{2^{k_m}}(\mathcal{U}_m)}$ is $2^{k_m+1}$-disjoint and $(2^{k_m+1}+B)$-bounded, by Lemma~\ref{partition2}, there exists a $2^{k_m+1}$-partition $L_1$ of $[0,8B]^{m+1}$ between $F_1^+$ and $F_1^-$ such that
 \[
 \text{$L_1\subseteq (\bigcup \overline{N_{2^{k_m}}(\mathcal{U}_m)})^c\cap [0,8B]^{m+1}$ and $d(L_1,F_1^{+/-})>2^{k_m+1}$,}
  \]
where $F_1^{+/-}$ is a pair of opposite facets of $[0,8B]^{m+1}$.
Since $L_1$ is a partition of $[0,8B]^{m+1}$ between $F_1^+$ and $F_1^-$, $[0,8B]^{m+1}=L_1\sqcup A_1 \sqcup B_1$ such that $A_1$, $B_1$ are open in $[0,8B]^{m+1}$ and $A_1$, $B_1$ contain two opposite facets $F_1^-$, $F_1^+$ respectively.\

Let $\mathcal{M}_1=\{Q\in \mathcal{Q}_1~|~Q\cap L_1\neq \varnothing\}$ and $M_1=\bigcup \mathcal{M}_1$. Since $L_1$ is a $2^{k_m+1}$-partition of $[0,8B]^{m+1}$ between $F_1^+$ and $F_1^-$, then  $M_1$ is a partition of $[0,8B]^{m+1}$ between $F_1^+$ and $F_1^-$, i.e., $[0,8B]^{m+1}=M_1\sqcup A'_1 \sqcup B'_1$ such that $A'_1$, $B'_1$ are open in $[0,8B]^{m+1}$ and $A'_1$, $B'_1$ contain two opposite facets $F_1^-$, $F_1^+$ respectively.
Let $L'_1=\partial_{m}M_1=\bigcup\{\partial_{m}Q|Q\in \mathcal{M}_1\}$, then $[0,8B]^{m+1}\setminus (L'_1\sqcup A'_1 \sqcup B'_1)$ is the union of some disjoint open $m+1$-dimensional cubes with length of edge $= 2^{k_m}$.
So $L_{1}'$ is a partition of $[0,8B]^{m+1}$  between $F_1^+$ and $F_1^-$ and $L'_1\subseteq (\bigcup \mathcal{U}_m)^c\cap [0,8B]^{m+1}$.\

Since $\overline{N_{2^{k_{m-1}}}(\mathcal{U}_{m-1})}$ is $2^{k_{m-1}+1}$-disjoint and $(2^{k_{m-1}+1}+B)$-bounded, there exists a $2^{k_{m-1}+1}$-partition $L_{2}$ of $L'_{1}$ between $F_{2}^+\cap L'_1$ and $F_{2}^-\cap L'_1$ such that
\[
\text{ $L_{2}\subseteq (\bigcup\overline{ N_{2^{k_{m-1}}}(\mathcal{U}_{m-1})})^c\cap L'_{1}$ and $d(L_{2},F_{2}^{+/-})>2^{k_{m-1}+1}$.}
\]
Since $L_{2}$ is a partition of $L'_{1}$ between $F_{2}^+\cap L'_1$ and $F_{2}^-\cap L'_1$, $L'_1=L_{2}\sqcup A_{2} \sqcup B_{2}$ such that $A_{2}$, $B_{2}$ are open in $L'_1$ and $A_{2}$, $B_{2}$ contain two opposite facets $F_{2}^-\cap L'_1$, $F_{2}^+\cap L'_1$ respectively.\

Assume that $p_{2}=\frac{8B}{2^{k_{m-1}}}\in \mathbb{N}$.
Take a bijection $\psi:\{1,2,\cdots,p_{2}^{m+1}\} \to \{0,1,2,\cdots,p_{2}-1\}^{m+1}$. Let
$$Q'(t,l)=\prod_{j=1}^l(2^{k_{m-1}}\psi(t)_j,2^{k_{m-1}}(\psi(t)_j+1))\times 2^{k_{m-1}}\psi(t)_l\times \prod_{j=l+1}^{m+1}(2^{k_{m-1}}\psi(t)_j,2^{k_{m-1}}(\psi(t)_j+1)),$$
where $\psi(t)_j$ is the $j$-th coordinate of $\psi(t)$.\
Let \[\mathcal{Q}_2=\{\overline{Q'(t,l)\cap L_1'}~~|~~t\in \{1,2,\cdots,p_{2}^{m+1}\},l\in\{1,\cdots,m+1\}\}\] be a family of $m$-dimensional cubes with length of edges=$2^{k_{m-1}}$, then $L'_{1}=\bigcup_{Q\in\mathcal{Q}_2}Q$.
Let\[\text{ $\mathcal{M}_{2}=\{Q'\in \mathcal{Q}_{2}~|~Q'\cap L_{2}\neq \varnothing\}$ and $M_{2}=\bigcup \mathcal{M}_{2}$. }\]Since $L_{2}$ is a $2^{k_{m-1}+1}$-partition of $L'_1$ between $F_{2}^+\cap L'_1$ and $F_{2}^-\cap L'_1$, then  $M_{2}$ is a partition of $L'_1$ between $F_{2}^+\cap L'_1$ and $F_{2}^-\cap L'_1$, i.e., $L'_1=M_{2}\sqcup A'_{2} \sqcup B'_{2}$ such that $A'_{2}$, $B'_{2}$ are open in $L'_1$ and $A'_{2}$, $B'_{2}$ contain two opposite facets $F_{2}^-\cap L'_1$, $F_{2}^+\cap L'_1$ respectively.
Let $L'_{2}=\partial_{m-1}M_{2}=\bigcup\{\partial_{m-1}Q'~|~Q'\in \mathcal{M}_{2}\}$. Then $L'_{1}\setminus (L'_{2}\sqcup A'_{2} \sqcup B'_{2})$ is the union of some disjoint open $m$-dimensional cubes with length of edge $= 2^{k_{m-1}}$.
So $L'_{2}$ is a partition of $L'_{1}$  between $F_{2}^+\cap L'_1$ and $F_{2}^-\cap L'_1$ and $L'_{2}\subseteq (\bigcup (\mathcal{U}_{m-1}\cup \mathcal{U}_m))^c\cap [0,8B]^{m+1}$.

After $m+1$ steps, we obtain $L'_{m+1}$ to be a partition of $L'_{m}$  between $F_{m+1}^+\cap L'_{m}$ and $F_{m+1}^-\cap L'_{m}$ and $L'_{m+1}\subseteq (\bigcup (\bigcup_{i=0}^m \mathcal{U}_i))^c\cap [0,8B]^{m+1}$.
Note that $L'_{m+1}$ is $0$-dimensional cubes with length of edge $= 2^{k_{0}}$.
By the construction, $L'_{m+1}\subseteq X_{\tau}$. Since $\bigcup_{i=0}^m\mathcal{U}_i$ covers $[0,8B]^{m+1}\cap X_{\tau}$, $L'_{m+1}=\varnothing$ which is a contradiction with Lemma \ref{partition2}.

\end{proof}

\begin{defi}
Let \[\text{$\bigoplus \mathbb{Z}=\{(x_i)~|~x_i\in\ZZ\text{ and there exists }k\in \mathbb{N} \text{ such that } x_j=0 \text{ for each }j\geq k\}$}\] with the sup-metric $\rho$.
For every $\tau=\{k_0+2,k_1+2,\cdots,k_m+2\}\in S_{\xi}[L]$, we define an isometric embedding $i_{\tau}: X_{\tau}\to \bigoplus \mathbb{Z}$ by
\[
i_{\tau}(x)_i=x_i~~\forall~ i\in \{0,1,\cdots,m\} \text{ and }i_{\tau}(x)_i=0~~\forall~ i\notin \{0,1,\cdots,m\},
\]
where $i_{\tau}(x)_i$ is the $i$-th coordinate of $i_{\tau}(x)$ and $x=(x_{0},x_{1},\cdots,x_{m})\in X_{\tau}$.

For any countable ordinal number $\xi$, we define $X_{\xi}$ as the disjoint union of $X_{\tau}$ with $\tau\in S_{\xi}[L]$, i.e.,
\[
X_{\xi}=\bigsqcup_{\tau\in S_{\xi}[L]}X_{\tau}
\]
with the metric $d_{\xi}$ which is defined as
\begin{eqnarray*}d_{\xi}(x,y)=
\begin{cases}
\rho(i_{\tau}(x),i_{\tau}(y)) & \text{ if }x,y\in X_{\tau},\cr
\max\{s(\tau_{1}),s(\tau_{2}),\rho(i_{\tau_{1}}(x),i_{\tau_{2}}(y))\} & \text{ if }x\in X_{\tau_{1}},y\in X_{\tau_{2}} \text{ and } \tau_{1}\neq\tau_{2},
\end{cases}
\end{eqnarray*}
where $s(\sigma)=2^{\max \sigma}$ for any $\sigma\in S_{\xi}[L]$.
\end{defi}

\begin{thm}\rm
\label{thmmain0}
 trasdim$(X_\xi)\ge\xi$.
\end{thm}
\begin{proof}
By Proposition \ref{prop1}, for every $\tau\in S_{\xi}[L]$, $\tau\in A_{2}(X_{\tau},d)\subseteq A_{2}(X_{\xi},d)$.
i.e., $S_{\xi}[L]\subseteq A_{2}(X_{\xi},d).$  So
\[
\text{trasdim}X_\xi=\text{Ord}A(X_{\xi},d)=\text{Ord}A_{2}(X_{\xi},d)\geq\text{Ord}S_{\xi}[L]=\xi.
\]
\end{proof}

\begin{thm}\rm\label{thmmain1} coasdim$(X_\xi)\le\xi$ for each countable ordinal number $\xi$.
\end{thm}

\begin{proof}
For finite $\xi$ it follows from Theorem 25 in \cite{BD}. We proceed by transfinite induction.
For $k\in\NN$ we put $T_k=\{\tau\in S_\xi[L]\mid |\tau|\geq n(\xi)+2 $ and $i(\tau,n(\xi))\ge k+2\}$ and $K_k=S_\xi[L]\setminus T_k$.\

Since $X_\tau\subset\{(x_1,\dots,x_{|\tau|})\in\mathbb{Z}^{|\tau|}\mid |\{j~|~x_j\notin 2^{k}\mathbb{Z}\}|\leq n(\xi)\}$ for each $\tau\in T_k$, by Lemma 3.8 in \cite{omega+k}, for any $r\in \NN$, there exists $p\ge2^r$ such that for each $\tau\in T_p$   there exist uniformly bounded $2^r$-disjoint  families $\U_{0},\dots,\U_{n(\xi)}$ which cover $X_\tau$. Since the distance between $X_\tau$ and $X_\sigma$ is greater or equal then $2^p$ for $\tau$, $\sigma\in T_p$, there exist uniformly bounded $2^r-$disjoint families which are still called $\U_{0},\dots,\U_{n(\xi)}$ such that $\bigcup_{i=0}^{n(\xi)}\U_i$ covers the set $Y=\coprod\{X_\tau\mid \tau\in T_p\}$.

So, we need to prove that   coasdim$(X_\xi\setminus Y)=$coasdim$(\coprod_{\tau\in K_p}X_\tau)<\gamma(\xi)$.  For each $\tau\in K_p$ we can consider  $X_\tau$ as a subset  of $X_{\zeta_l(\gamma(\xi))+l+n(\xi)+1}$ for some $l\in\{1,\dots,p+1\}$ by Lemma \ref{lemsmall} and Remark 3.2. Moreover,
we have that the natural embedding $e_\tau:X_\tau\to X_{\zeta_l(\gamma(\xi))+l+n(\xi)+1}$ is isometrical by the definition of the metrics on  $X_\tau$ and  on $X_{\zeta_l(\gamma(\xi))+l+n(\xi)+1}$. Let us define a metric space $(Z,d)$. Put  $Z=\coprod_{l\in\{1,\dots,p+1\}}X_{\zeta_l(\gamma(\xi))+l+n(\xi)+1}$. Define the metric $d$ as follows. On each  $X_{\zeta_l(\gamma(\xi))+l+n(\xi)+1}$  we consider its metric  $d_{\zeta_l(\gamma(\xi))+l+n(\xi)+1}$. If $x\in X_{\zeta_l(\gamma(\xi))+l+n(\xi)+1}$ and $y\in X_{\zeta_s(\gamma(\xi))+s+n(\xi)+1}$ with $s\ne l$, we have that $x\in X_\tau$ and $y\in X_\sigma$ for some $\tau\in S_{\zeta_l(\gamma(\xi))+l+n(\xi)+1}$ and $\sigma\in S_{\zeta_s(\gamma(\xi))+s+n(\xi)+1}$. Then we put $d(x,y)=\max\{s(\tau),s(\sigma),\rho(i_\tau(x),i_\sigma(y))\}$ where  $\rho$ is the sup-metric in $\bigoplus\ZZ$.
 By Lemma \ref{asdimunion} we have
 \[
 \text{coasdim}(\coprod_{l=1}^{p+1}X_{\zeta_l(\gamma(\xi))+l+n(\xi)+1})\leq\max_{l\in\{1,\dots,p+1\}}
 \zeta_{p+1}(\gamma(\xi))+l+n(\xi)+1\leq\zeta_{p+1}(\gamma(\xi))+p+n(\xi)+2<\gamma(\xi).
 \]

Define the map $e:\coprod_{\tau\in K_p}X_\tau\to Z$  by the conditions $e(x)=e_\tau(x)$ for each $x\in X_\tau$. We have that $e$ is an isometrical embedding.  Hence  coasdim$(\coprod_{\tau\in K_p}X_\tau)\le$coasdim$Z<\gamma(\xi)$.

\end{proof}

\begin{lem}\rm(\cite{omega+k})
\label{lemrelation}
Let $X$ be a metric space, if $X$ has complementary-finite asymptotic dimension, then trasdim$(X)\leq\text{coasdim}(X)$.

\end{lem}

\begin{thm}\rm
coasdim$(X_{\xi})=\xi$ and trasdim$(X_{\xi})= \xi$ for every countable ordinal $\xi$.
\end{thm}
\begin{proof}
By Theorem \ref{thmmain1} and Lemma \ref{lemrelation}, trasdim$(X_{\xi})\leq \xi$.
Then trasdim$(X_{\xi})= \xi$ by Theorem \ref{thmmain0}.
By Lemma \ref{lemrelation}, coasdim$(X_{\xi})\geq$trasdim$(X_{\xi})= \xi$.
Then coasdim$(X_{\xi})=\xi$ by Theorem \ref{thmmain1}.

\end{proof}

\end{section}

\begin{section}{Acknowledgements}
The authors wish to thank the reviewers for careful reading and valuable comments. This work was
supported by NSFC grant of P.R. China (No. 12071183, 11871342).
And the authors (Yan Wu and Jingming Zhu) want to thank V.M. Manuilov and Benyin Fu for helpful discussions.
The third author (Taras Radul) is grateful to Taras Banakh for valuable and stimulating discussions.
\end{section}


\begin{thebibliography}{10}
\bibitem{Gromov}
M.~Gromov, \emph{Asymptotic invariants of infinite groups.} in: Geometric Group Theory, Vol.2, Sussex, 1991, in: Lond. Math. Soc. Lect. Note Ser., vol.182, Cambridge Univ. Press, Cambridge, (1993), 1--295.

\bibitem{Dranishnikov}
A. Dranishnikov, \emph{Asymptotic topology.} Russ. Math. Surv. 55 (2000) 1085--1129.




\bibitem{Radul2010}
T.~Radul, \emph{On transfinite extension of asymptotic dimension.} Topol. Appl. 157 (2010), 2292--2296.



\bibitem{satkiewicz}
M.~Satkiewicz, \emph{Transfinite Asymptotic Dimension.} arXiv:1310.1258v1,
2013.


\bibitem{omega+1}
Jingming Zhu, Yan Wu, \emph{A metric space with its transfinite asymptotic dimension $\omega+1$.} Topol. Appl.(2020),https://doi.org/10.1016/j.topol.2020.107115.



\bibitem{omega+k}
Jingming Zhu, Yan Wu, \emph{Examples of metric spaces with asymptotic property C.} arxiv.org:1912.02103




\bibitem{Orz2020}
K. Orzechowski, \emph{APD profiles and transfinite asymptotic dimension.} Topol. Appl. 283(2020) 107394

\bibitem{Dydak2020}
J. Dydak, \emph{Matrix algebra of sets and variants of decomposition complexity.} Rev. Mat. Complut. 33(2020) 373--388.

%\bibitem{Yu2012}
%E. Guentner, R. Tessera, G. Yu, \emph{A notion of geometric complexity and its application to topological rigidity.} Invent. Math. 189(2012), 315--357.
%
%\bibitem{Yu2013}
%E. Guentner, R. Tessera, G. Yu, \emph{Discrete groups with finite decomposition complexity.} Groups Geom. Dyn. 7 (2013), 377--402.
%
%
%\bibitem{Mari2014}
%A. Dranishnikov, M. Zarichnyi,\emph{ Asymptotic dimension, decomposition complexity, and Havar¡¯s property C.} Topol. Appl. 169(2014), 99--107.












\bibitem{Bell2011}
G.~Bell, A. Dranishnikov, \emph{Asymptotic dimension in Bedlewo.} Topol. Proc. 38 (2011), 209--236.

\bibitem{BD}
G.~Bell, A. Dranishnikov, \emph{Asymptotic dimension.} Topology and its Applications 155 (2008) 1265--296

\bibitem{Borst1988}
P.~Borst, \emph{Classification of weakly infinite-dimensional spaces.} Fund. Math. 130(1988), 1--25.




\bibitem{yanzhu2018}
Y. Wu, J. Zhu, \emph{Classification of metric spaces with infinite asymptotic
dimension.} Topol. Appl. 238 (2018), 90--101.






\bibitem{Engelking}
R.~Engelking, \emph{Theory of Dimensions: Finite and Infinite.} Heldermann Verlag, 1995.



























\end{thebibliography}
\end{document}